\documentclass[leqno]{amsart}
\usepackage{amsmath,amsfonts,amsthm,amssymb,indentfirst,epic,url}

\newtheorem{theorem}{Theorem}
\newtheorem{lemma}[theorem]{Lemma}
\newtheorem{corollary}[theorem]{Corollary}
\newtheorem{proposition}[theorem]{Proposition}
\newtheorem{definition}[theorem]{Definition}
\newtheorem{question}[theorem]{Question}
\newtheorem{example}[theorem]{Example}

\newcommand{\lm}{\varprojlim}
\newcommand{\length}{\mathrm{length}}
\newcommand{\V}{\mathbf{V}}

\begin{document}

\title%
{Continuity of homomorphisms on pro-nilpotent algebras}
\thanks{Any updates, errata, related references, etc., learned of
after publication of this note will be recorded at
\url{http://math.berkeley.edu/~gbergman/papers/}\,.
}

%      top.rings Lie"&" Jord"&"

\author{George M. Bergman}
\address{University of California\\
Berkeley, CA 94720-3840, USA}
\email{gbergman@math.berkeley.edu}

\begin{abstract}
Let $\V$ be a variety of not necessarily associative algebras,
and $A$ an inverse limit of nilpotent algebras $A_i\in\V,$ such
that some finitely generated subalgebra $S\subseteq A$ is dense in $A$
under the inverse limit of the discrete topologies on the $A_i.$

A sufficient condition on $\V$ is obtained for all
algebra homomorphisms from $A$ to finite-dimensional algebras $B$ to be
continuous; in other words, for the kernels of all such
homomorphisms to be open ideals.
This condition is satisfied, in particular, if $\V$ is the variety
of associative, Lie, or Jordan algebras.

Examples are given showing the need for our hypotheses, and some open
questions are noted.
\end{abstract}
\maketitle
% - - - - - - - - - - - - - - - - - - - - - - - - - - - - - -

\section{Background: From pro-$\!p\!$ groups to pro-nilpotent algebras}\label{S.intro}

A result of Serre's on topological groups says that if $G$ is a
pro-$\!p\!$ group (an inverse limit of finite $\!p$-groups)
which is topologically finitely generated
(i.e., has a finitely generated subgroup which is dense in
$G$ under the inverse limit topology), then any homomorphism from $G$
to a finite group $H$ is continuous (\cite[Theorem~1.17]{pro-p}
\cite[\S I.4.2, Exercises~5-6, p.32]{Serre}).
Two key steps in the proof are that~ (i)~every finite homomorphic
image of a pro-$\!p\!$ group $G$ is
a $\!p$-group, and (ii)~for $G$ a topologically finitely
generated pro-$\!p\!$ group, its subgroup $G^p\,[G,G]$ is closed.

In \cite{pro-np} we obtained a result similar to (i),
namely, that if $A$ is a pro-nilpotent (not necessarily associative)
algebra over a field $k,$ then every finite-dimensional
homomorphic image of $A$ is nilpotent.
The analog of (ii) would say that when such an $A$ is
topologically finitely generated, the ideal $A^2$ is closed.
(To see the analogy, note that $G/(G^p\,[G,G])$ is the universal
homomorphic image of $G$ which is a
$\!\mathbb{Z}/p\mathbb{Z}$-vector space, and
$A/A^2$ the universal homomorphic image of $A$ which is a
$\!k\!$-vector space with zero multiplication.)

Is this analog of (ii) true?

The proof of the group-theoretic statement (ii) is based on
first showing that if a finite $\!p$-group $F$
is generated by $g_1,\dots,g_r,$ then
\begin{equation}\begin{minipage}[c]{23pc}\label{d.[]=}
$[F,F]\ =\ [F,g_1]\dots [F,g_r],$
\end{minipage}\end{equation}
i.e., every member of $[F,F]$ is
a product of exactly $r$ commutators of the indicated sorts.
From this one deduces that if a pro-$\!p\!$ group $G$ is
generated topologically by $g_1,\dots,g_r,$ then likewise
$G^p\,[G,G]=G^p\,[G,g_1]\dots [G,g_r].$
The latter set is compact since $G$ is, hence it must be closed in $G,$
as claimed.

If an {\em associative} algebra $S$ is generated by elements
$g_1,\dots,g_r,$ it is clear that, similarly,
\begin{equation}\begin{minipage}[c]{23pc}\label{d.S^2=}
$S^2\ =\ S\,g_1+\dots+S\,g_r\,.$
\end{minipage}\end{equation}
We shall see below that whether a formula
like~(\ref{d.S^2=}) holds for all finitely generated
$S$ in a variety $\V$ of not necessarily associative
algebras depends on $\V.$
In particular, we shall find that~(\ref{d.S^2=}) itself also holds
for Lie algebras, while a more complicated relation which
we can use in the same way holds for Jordan algebras.
For varieties in which such identities hold, we obtain an analog
of Serre's result, Theorem~\ref{T.main} below.
Examples in \S\ref{S.cegs} show the need for some such condition on
$\V,$ and for the topological finite generation of $A.$

I am indebted to
Yiftach Barnea
for inspiring this note
by pointing to the analogy between the main result of \cite{pro-np},
and step~(i) in the proof of Serre's result on pro-$\!p\!$ groups.
I am also grateful to J.-P.\,Serre and
E.\,Zelmanov for helpful answers to questions I sent them, and
to the referee for some useful suggestions.

\section{Review of linearly compact vector spaces}\label{S.lin_cp}
The analog of the compact topology on a profinite group is the
{\em linearly compact} topology on (the underlying
vector-space of) a pro-finite-dimensional algebra.
Most of the basic arguments regarding linear compactness of
topological vector spaces work
for left modules over a general associative ring, and
we shall sketch the material here in that context.
However we shall only call on the vector space case, so the reader who
so prefers may read this section with only that case in mind.

\begin{definition}\label{D.lin_cp}
Let $M$ be a left module over an associative unital ring $R.$

A {\em linear topology} on $M$ means a topology under which
the module operations are continuous, and which has
a neighborhood basis of $0$ consisting of submodules.
\textup{(}A basis of open sets is then given by the cosets
of these open submodules.\textup{)}

Under such a topology, $M$ is said to be {\em linearly compact} if
it is Hausdorff, and if every family of cosets of closed submodules of
$M$ that has the finite intersection property has nonempty intersection.
\end{definition}

(We assume no topology given on $R.$
Indeed, even if $R$ is the real or complex field, its standard topology
is unrelated to the linear topologies on $\!R\!$-vector-spaces.)

The closed submodules, used in the above definition of
linear compactness, are characterized in

\begin{lemma}\label{L.closed=cap_open}
Under any linear topology on a module $M,$ the closed submodules
are the \textup{(}not necessarily finite\textup{)}
intersections of open submodules.
\end{lemma}

\begin{proof}
Every open submodule $N$ is closed, since its complement is
the union of its nonidentity cosets, which are open.
Hence intersections of open submodules are also closed.

Conversely, if $N$ is a closed submodule and $x$ a point not in $N,$
then for some open submodule $M',$ the coset $x+M'$ is disjoint
from $N;$ equivalently, $x\notin N+M'.$
Since the submodule $N+M'$ is a union of cosets of $M',$
it is open, so $N$ is
the intersection of the open submodules containing it.
\end{proof}

We shall see below that the closed submodules of a linearly
topologized module do not in general
determine the open submodules, and hence the topology; but
that they do when $R$ is a field and $M$ is linearly compact.

Here are some tools for proving a linear topology linearly compact.

\begin{lemma}[cf.\ {\cite[II\,(27.2-4),\,(27.6)]{Lefschetz}}]\label{L.lcp<=}
Let $R$ be an associative unital ring.\vspace{.2em}

\textup{(i)}~ If an $\!R\!$-module $M$ with a Hausdorff
linear topology has
descending chain condition on closed submodules, it is linearly compact.
In particular, if an $\!R\!$-module $M$ is artinian, it is
linearly compact under the discrete topology.\vspace{.2em}

\textup{(ii)}~ A closed submodule of a linearly compact
$\!R\!$-module is linearly compact under the induced
topology.\vspace{.2em}

\textup{(iii)}~ The image of a linearly compact module under
a continuous homomorphism into a Hausdorff linearly topologized module
is again linearly compact.
In particular, any Hausdorff linear topology on a module weaker
than a linearly compact topology is linearly compact.\vspace{.2em}

\textup{(iv)}~ The limit \textup{(}in the
category-theoretic sense, which includes inverse limits,
direct products, fixed-point modules of group actions,
etc.\textup{)}\ of any small system of
linearly compact $\!R\!$-modules and continuous maps among
them is again linearly compact.
\textup{(}``Small'' means indexed by a set rather than a proper
class.\textup{)}
\end{lemma}

\begin{proof}[Sketch of proof]
The verifications of (i)-(iii) are routine.
(The Hausdorffness condition in (iii) is needed only because
``Hausdorff'' is part of the definition of linearly compact.)

To get (iv), recall \cite[proof of Theorem~V.2.1, p.109]{CW}
\cite[proof of Proposition~7.6.6]{245} that if a category
has products and equalizers, then it has small limits, and
the limit of a small system of objects $A_i$ and morphisms among
them can be constructed as the equalizer of a pair of maps
between products of copies of the $A_i.$
Now the product topology on a direct product of linearly
topologized modules is again linear,
and the proof of Tychonoff's theorem adapts to show that a
direct product of linearly compact modules is again linearly compact
(cf.\,\cite[II\,(27.2)]{Lefschetz}).
The equalizer of two morphisms in the category of linearly
topologized modules is the kernel of their difference,
a submodule which is closed if the codomain of the maps is Hausdorff.
In this situation,~(ii) shows
that if the domain module is linearly compact, the equalizer
is also linearly compact, completing the proof of~(iv).
\end{proof}

Note that by the second sentence of~(i) above, every finite-dimensional
vector space over a field $k$ is linearly compact under the
discrete topology.
In particular, linear compactness does not imply
ordinary compactness.
Bringing in~(iv), we see that an inverse limit of
discrete finite dimensional vector spaces is always linearly compact.

Over a general ring $R,$ are the artinian modules the only modules
linearly compact in the discrete topology?
Not necessarily: if $R$ is a complete discrete valuation ring which is
not a field, we see that as a
discrete $\!R\!$-module, $R$ is linearly compact.
(More on this example later.)

Here are some restrictions on linearly compact topologies.

\begin{lemma}[cf.\,{\cite[II\,(25.6),\,(27.5),\,(27.7)]{Lefschetz}}]\label{L.lcp=>}
Let $R$ be an associative unital ring.
\vspace{.2em}

\textup{(i)}~ If an $\!R\!$-module $M$ is artinian, then
the only Hausdorff linear topology on $M$ is the discrete
topology.\vspace{.2em}

\textup{(ii)}~ If an $\!R\!$-module $M$ is linearly compact
under the discrete topology, then $M$ does not contain a direct
sum of infinitely many nonzero submodules.
\vspace{.2em}

\textup{(iii)}~ Any linearly compact submodule $N$ of a linearly
topologized $\!R\!$-module $M$ is closed.
\vspace{.2em}

\textup{(iv)}~ Every linearly compact $\!R\!$-module
$M$ is the inverse limit of an inversely directed system of
discrete linearly compact $\!R\!$-modules.
\vspace{.2em}

\textup{(v)}~ In a linearly compact $\!R\!$-module, the sum
of any two closed submodules is closed.
\end{lemma}

\begin{proof}[Sketch of proof]
(i): Hausdorffness implies that $\{0\}$ is the intersection of
all the open submodules of $M.$
By the artinian assumption, some finite intersection of these
is therefore zero.
A finite intersection of open submodules is open,
so $\{0\}$ is open, i.e., $M$ is discrete.

(ii):  If $M$ contains an infinite direct sum
$\bigoplus_{i\in I} N_i,$ and
each $N_i$ has a nonzero element $x_i,$ then the system of cosets
$C_i=x_i+\bigoplus_{j\in I-\{i\}} N_j$ $(i\in I)$ has the finite
intersection property (namely, $C_{i_1}\cap\dots\cap C_{i_n}$
contains $x_{i_1}+\dots+x_{i_n}),$ and all these sets are closed since
$M$ is discrete.
But they have no element in common:
such an element would lie in $\bigoplus_{i\in I} N_i$
since each $C_i$ does, but
would have to have a nonzero summand in each $N_i.$

(iii):  Suppose $x\in M$ is in the closure of the linearly
compact submodule $N.$
Then for every open submodule $M'\subseteq M,$ the coset
$x+M'$ has nonempty intersection with $N,$ and we see that
these intersections form a family of cosets within $N$ of
the submodules $N\cap M',$ which are clearly closed in $N.$
Linear compactness of $N$ implies that these sets have nonempty
intersection; but the intersection of the larger sets $x+M'$ is $\{x\}$
because $M$ is Hausdorff.
Hence $x\in N,$ showing that $N$ is closed.

(iv):  Note that the open submodules $N\subseteq M$ form an inversely
directed system under inclusion; let $M'=\lm_N\,M/N,$ the inverse limit
of the system of discrete factor-modules, with the
inverse-limit topology.
The universal property of the inverse limit gives us a
continuous homomorphism $f:M\to\nolinebreak M'.$

Now each point of $M'$ arises from a system of elements
of the factor-modules $M/N,$ equivalently, from a system
of cosets of the open submodules $N,$ having a compatibility relation
that implies the finite intersection property.
Using the linear compactness of $M,$ we deduce that $f$ is surjective.
Since $M$ is Hausdorff, the maps $M\to M/N,$ as $N$ ranges over the
open sets, separate points,
hence so does the single map $f:M\to M';$ so $f$ is also injective.

Finally, every open submodule $N$ of $M$ is the
inverse image of the open submodule $\{0\}\subseteq M/N,$
hence is the inverse image of an open submodule of $M';$ so
the topology of $M$ is no finer than that of $M',$
so $f$ is an isomorphism of topological modules.

(v):  Let $M_1$ and $M_2$ be closed submodules of the
linearly compact module $M.$
By~(iv) of the preceding lemma, the direct product
$M_1\times M_2$ is linearly compact under the product topology.
The map $M_1\times M_2\to M$ given by addition is continuous,
hence its image, $M_1+M_2,$ is linearly compact by point~(iii) of
that lemma, hence closed by~(iii) of this one.
\end{proof}

Returning to the curious case of a complete discrete valuation
ring $R,$ which we saw was linearly compact under the discrete topology,
one may ask, ``What about its valuation topology, which is
not discrete?''
The next lemma (not needed for the main results of this paper)
shows that $R$ is linearly compact under that topology as well.

\begin{lemma}\label{L.weaken}
Let $R$ be an associative unital ring, $M$ a left $\!R\!$-module,
and $T'\subseteq T$ Hausdorff linear topologies on $M.$

Then if $M$ is linearly compact under $T,$ it is also
linearly compact under $T'.$
In this situation, the same submodules \textup{(}but not,
in general, the same sets\textup{)} are closed in the two topologies.
\end{lemma}

\begin{proof}
The first assertion is the content
of the second sentence of Lemma~\ref{L.lcp<=}(iii).

Now if a submodule $N\subseteq M$ is closed under $T,$
then by Lemma~\ref{L.lcp<=}(ii) it is linearly compact
under the topology induced on it by $T,$ hence, by
the above observation with $N$ in place of $M,$ also under
the weaker topology induced by $T'.$
Hence by Lemma~\ref{L.lcp=>}(iii),
it is closed in $M$ under $T'.$
The reverse implication follows from the assumed
inclusion of the topologies, giving the second assertion of the lemma.

For $R$ a complete discrete valuation ring which is not
a field, and $M=R,$
observe that $M-\{0\}$ is closed in the discrete topology but
not in the valuation topology, yielding the parenthetical
qualification.
\end{proof}

We now note some stronger statements that are true when $R$ is a field.
(The proofs of (i), (ii) and (iv) below work, with appropriately
adjusted language, for arbitrary $R$ if ``vector space'' is changed
to ``semisimple module'', i.e., direct sum of simple modules.)

\begin{lemma}[cf.\,{\cite[II\,(27.7), (32.1)]{Lefschetz}}]\label{L.lncp/fd}
Let $k$ be a field and $V$ a topological $\!k\!$-vector-space.
Then
\vspace{.2em}

\textup{(i)}~ If $V$ is discrete, it is linearly compact if
and only if it is finite-dimensional.
\vspace{.2em}

\textup{(ii)}~ If $V$ is linearly compact, then its open subspaces
are precisely its closed subspaces of finite
codimension.\vspace{.2em}

\textup{(iii)}~ The following conditions are equivalent:
\textup{(a)}~ $V$ is linearly compact.
\textup{(b)}~ $V$ is the inverse limit of an inversely directed
system of finite-dimensional discrete vector spaces.
\textup{(c)}~ Up to isomorphism of topological vector spaces,
$V$ is the direct product $k^I$ of a family of copies of $k,$
each given with the discrete topology.
\vspace{.2em}

\textup{(iv)}~ If $V$ is linearly compact, then no strictly
weaker or stronger topology on $V$ makes it linearly compact.
Equivalently, no linear topology strictly weaker than a
linearly compact topology is Hausdorff.
\end{lemma}

\begin{proof}
(i): ``If'' holds by Lemma~\ref{L.lcp<=}(i),
``only if'' by Lemma~\ref{L.lcp=>}(ii).

(ii): This follows from the fact that a submodule $N$
of a linearly topologized module $M$ is open if and only if
it is closed and $M/N$ is discrete, together with (i).

(iii):  (a)$\!\implies\!$(b) holds by Lemma~\ref{L.lcp=>}(iv),
and~(i) above.
To see (b)$\!\implies\!$(c), note that a discrete finite-dimensional
vector space is a finite
direct product of copies of $k$ under its discrete topology,
and that a direct product of direct products is a direct product.
(c)$\!\implies\!$(a) follows from Lemma~\ref{L.lcp<=}(iv),\,(i).

(iv):  If $T$ is a linearly compact topology on $V,$ then under any
Hausdorff linear topology $T'\subseteq T,$ Lemma~\ref{L.weaken} tells
us that $V$ will again
be linearly compact, with the same closed subspaces.
Hence the same subspaces will be closed of finite codimension,
i.e., by (ii), open; so the topologies are the same.
This gives the second sentence of~(iv), which in view
of Lemma~\ref{L.weaken} is equivalent to the first.
\end{proof}

For the remainder of this note we will study algebras,
assuming our base ring is a field $k,$ though many of the arguments
could be carried out for more general commutative base rings.
We record the straightforward result:

\begin{lemma}\label{L.XA+AY}
Let $A=\lm_i A_i$ be the inverse limit of an inversely directed
system of $\!k\!$-algebras.
Then the multiplication of $A$ is bicontinuous
in the inverse limit topology.\qed
\end{lemma}

(A very important aspect of the theory of linearly compact
vector spaces which does not come into this note is
the duality between the category of such spaces and the category
of discrete vector spaces; cf.\ \cite[Proposition~24.8]{coalg}.
More generally, \cite[(29.1)]{Lefschetz} establishes a
{\em self}\/-duality
for the category of {\em locally} linearly compact spaces, i.e.,
extensions of linearly compact spaces by discrete spaces.)

\section{Nilpotent algebras and $\!m\!$-separating monomials}\label{S.separate}
Let $A$ be an algebra over a field $k.$
If $B$ and $C$ are $\!k\!$-subspaces
of $A,$ we denote by $BC$ the $\!k\!$-subspace
spanned by all products $bc$ $(b\in B,\,c\in C).$

We define recursively $\!k\!$-subspaces
$A_{(n)}$ $(n=1,2,\dots)$ of $A$ by
\begin{equation}\begin{minipage}[c]{23pc}\label{d.A_(n)}
$A_{(1)}\,\ =\ A,\qquad A_{(n+1)}\ =
\ \sum_{0<m<n+1}\,A_{(m)}\,A_{(n+1-m)}.$
\end{minipage}\end{equation}

It is easy to see by induction (without assuming $A$
associative!)\ that for $n>0,$ $A_{(n+1)}\subseteq A_{(n)}.$
These subspaces are ideals, since
$A_{(n)}A=A_{(n)}A_{(1)}\subseteq A_{(n+1)}\subseteq A_{(n)},$
and similarly $A\,A_{(n)}\subseteq A_{(n)}.$
The algebra $A$ is said to be {\em nilpotent} if $A_{(n)}=\{0\}$
for some $n\geq 1.$
(Some other formulations of the condition of nilpotence,
which we will not need here,
are shown equivalent to this one in \cite[\S4]{pro-np}.)
\vspace{.5em}

Let me now preview the proof of our main result in an easy
case, that of associative algebras.

Suppose $A$ is an inverse limit of finite-dimensional associative
$\!k\!$-algebras, and that some finitely generated subalgebra
$S\subseteq A,$ say generated by $g_1,\dots,g_r,$
is dense in $A$ under the inverse limit topology.

For any $n>1,$ every element of $S_{(n)}$ is a
linear combination of monomials of lengths $N\geq n$ in the given
generators, each of which may be factored
$a\,g_{i_1}\dots\,g_{i_{n-1}},$
where $a$ is a product of $N-(n-1)$ generators.
It follows that
\begin{equation}\begin{minipage}[c]{23pc}\label{d.Sgn}
$S_{(n)}\ =\ \sum_{i_1,\dots,i_{n-1}\in\{1,\dots,r\}}
\ S\,g_{i_1}\dots\,g_{i_{n-1}}.$
\end{minipage}\end{equation}
Also, $S$ is spanned modulo $S_{(n)}$ by the finitely many
monomials in $g_1,\dots,g_r$ of lengths $<n;$ hence
by Lemma~\ref{L.lcp=>}(v), $A,$ the closure of $S,$ is spanned
modulo the closure of $S_{(n)}$ by those same monomials.
In particular,
\begin{equation}\begin{minipage}[c]{23pc}\label{d.clSgn}
The closure of~(\ref{d.Sgn}) in $A$ has finite codimension in~$A.$
\end{minipage}\end{equation}

Now consider
\begin{equation}\begin{minipage}[c]{23pc}\label{d.Agn}
$\sum_{i_1,\dots,i_{n-1}\in\{1,\dots,r\}}
\ A\,g_{i_1}\dots\,g_{i_{n-1}}.$
\end{minipage}\end{equation}
Since the maps $a\mapsto a\,g_{i_1}\dots g_{i_{n-1}}$ are continuous,
the above sum is closed in $A$
(by Lemmas~\ref{L.lcp<=}(iii), \ref{L.lcp=>}(iii)
and~\ref{L.lcp=>}(v)); and it obviously
contains~(\ref{d.Sgn}) and is contained in $A_{(n)}.$
So by~(\ref{d.clSgn}),
$A_{(n)}$ contains a closed subspace of finite codimension
in $A;$ hence it is open by Lemma~\ref{L.lncp/fd}(ii).

Suppose, now, that $f$ is a homomorphism from $A$ to a nilpotent
discrete algebra.
Then $\ker f$ must contain some $A_{(n)},$
hence will be open, hence $f$ will be continuous.

Finally, by the result from \cite{pro-np} mentioned in the
Introduction,
if $A$ is an inverse limit of nilpotent algebras, then any
finite-dimensional homomorphic image of $A$ {\em is} nilpotent;
so in that case, any
homomorphism of $A$ to a finite-dimensional algebra is
continuous.\vspace{.5em}

A key aspect of the above argument was that we were able to express
the general length-$\!N\!$ monomial $(N\geq n)$ in our generators
as the image of a general monomial $a$ under one of a fixed
finite set of linear operators (in this case, those of the form
$a\mapsto a\,g_{i_1}\dots\,g_{i_{n-1}})$
defined using multiplications by our generators.
For an arbitrary finitely generated
not-necessarily-associative algebra, no such
decomposition is possible, and we will see that
our main result does not apply to algebras in arbitrary varieties.
What we shall show next, however, is that {\em if} the identities
of our algebra allow us to handle, in roughly this way,
elements of $S_{(2)}=SS,$ then, as
above, we can do the same for all $S_{(n)}.$

We need some terminology.
By a {\em monomial} we shall mean an expression representing
a bracketed product of indeterminates.
E.g., $(xy)z$ and $x(yz)$ are distinct monomials.
(Since our algebras are not unital, we do not allow an empty monomial.
We do allow monomials involving repeated indeterminates.)
The {\em length} of a monomial will mean its length
as a string of letters, ignoring parentheses;
e.g., $\length((xy)z)=3.$

Note that every monomial $w$ of length $>1$ is in a unique way
a product of two monomials, $w=w'w''.$
Let us define recursively the {\em submonomials} of a monomial $w:$
Every monomial is a submonomial of itself,
and if $w=w'w'',$ then the submonomials of $w$
other than $w$ are the submonomials of $w'$ and the
submonomials of $w''.$
For example, the submonomials of $(xy)z$ include $xy,$ but not $yz.$

We now define a technical concept that we shall need.

\begin{definition}\label{D.separating}
If $w$ is a monomial and $m$ a natural number, we shall
say that $w$ is {\em $\!m\!$-separating} if $w$ has a
submonomial of length exactly $\length(w)-m.$

If $n\leq N$ are natural numbers, we shall call
$w$ {\em $\![n,N]\!$-separating} if it is $\!m\!$-separating
for some $m\in[n,N]=\{n,n+1,\dots,N\}.$
\end{definition}

For example, note that $((x_1 x_2)(x_3 x_4))((x_5 x_6)(x_7 x_8))$
has submonomials only of lengths $8,4,2$ and $1,$ hence it is
$\!m\!$-separating only for $m=0,4,6,7;$
in particular, it is not $\![1,3]\!$-separating.
On the other hand, $(((x_1 x_2)(x_3 x_4))((x_5 x_6)(x_7 x_8)))\,x_9,$
being of length $9$ and having a submonomial of length $8,$ is
$\!1\!$-separating, hence it is $\![1,3]\!$-separating.

Recall that a {\em variety} of algebras means the class of
all algebras satisfying a fixed set of identities.
An identity for algebras may be written as saying that a certain
linear combination of monomials evaluates to zero.
Such an identity is called {\em homogeneous} if
the monomials in question all have the
same number of occurrences of each variable.
(It is easy to show that any variety of algebras over
an infinite field is determined by homogeneous identities.
An example of a variety of algebras over a finite field which
is not so determined is that of Boolean rings,
regarded as algebras over $\mathbb{Z}/2\mathbb{Z};$
the identity $x^2=x$ is not a consequence of homogeneous identities
of that variety.
In this note, only homogeneous identities will interest us.)

Note that the variety of {\em associative} algebras has identities
which equate every monomial with a \mbox{$\!1\!$-separating}
monomial.
From this one can obtain identities
which equate every monomial of length $\geq n$ with an
$\!n\!$-separating monomial; this fact underlies the sketch just given
of the proof of our main result for associative algebras.
Here is the analogous relationship for general varieties.

\begin{lemma}\label{L.n+d-sep}
Suppose $\V$ is a variety of algebras, and $d$ a positive
integer, such that for every monomial $w$ of
degree $>1,$ $\V$ satisfies
a homogeneous identity equating $w$ with a $\!k\!$-linear
combination of $[1,1{+}d\,]\!$-separating monomials.

Then for every positive integer $n$ and every monomial $w$
of length $>n,$ $\V$ satisfies a homogeneous identity equating
$w$ with a $\!k\!$-linear combination of $\![n,n{+}d\,]\!$-separating
monomials.
\end{lemma}

\begin{proof}
The case $n=1$ is the hypothesis.
Given $n>1,$ assume inductively that the result is true
for $n-1,$ and let $w$ be a monomial of length $>n.$

By our inductive assumption, $w$ is congruent modulo
homogeneous identities of $\V$ to a linear combination of
$\![n{-}1,\,n{-}1{+}d\,]\!$-separating monomials.
By homogeneity of the identities in question, these monomials can
all be assumed to have length equal to $\length(w).$
Suppose $w'$ is one of these monomials which is not
$\![n,n{+}d\,]\!$-separating.
Since it {\em is} $\![n{-}1,\,n{-}1{+}d\,]\!$-separating, it must
be $\!(n{-}1)\!$-separating; so it has a submonomial $w''$ of
length $\length(w)-(n-1).$
By our original hypothesis, this $w''$ is congruent modulo homogeneous
identities of $\V$ to a linear combination of
$\![1,1{+}d\,]\!$-separating monomials.
If we substitute this expression for the submonomial $w''$ into
$w',$ we get an expression for the latter as a linear combination
of $\![1{+}(n{-}1),1{+}d{+}(n{-}1)\,]\!$-separating,
i.e., $\![n,n{+}d\,]\!$-separating monomials.
Doing this for each such $w',$ we get the desired expression
for $w$ modulo the identities of $\V.$
\end{proof}

The hypothesis of the preceding lemma will be used throughout the
remainder of this note, so let us give it a name.

\begin{definition}\label{D.sep}
We shall call a variety $\V$ of $\!k\!$-algebras
{\em $\![1,1{+}d\,]\!$-separative} if every monomial of
degree $>1$ is congruent modulo homogeneous identities of $\V$ to a
$\!k\!$-linear combination of $[1,1{+}d\,]\!$-separating monomials.
We shall call $\V$ {\em separative} if it is
$\![1,1{+}d\,]\!$-separative for some natural number $d.$
\end{definition}

\section{The main theorem}\label{S.main}

The proof of the next result, our main theorem,
follows the outline sketched at
the beginning of the preceding section; but we will give details,
and use Lemma~\ref{L.n+d-sep} in place of familiar properties
of associativity.

\begin{theorem}\label{T.main}
Let $\V$ be a separative variety of $\!k\!$-algebras, and
$A\in\V$ an inverse limit of finite-dimensional
algebras; and suppose $A$ has a finitely generated subalgebra $S$
which is dense in the inverse limit topology on $A.$
Then
\vspace{.2em}

\textup{(i)}~ For all $n>0,$ the ideal $A_{(n)}$ is open in the
inverse limit topology on $A.$
Hence,
\vspace{.2em}

\textup{(ii)}~ Every homomorphism of $A$ into a nilpotent
$\!k\!$-algebra $B$ is continuous \textup{(}with respect to
the discrete topology on $B).$
Hence, by \textup{\cite[Theorem~10(iii)]{pro-np}},
\vspace{.2em}

\textup{(iii)}~ If the algebras of which $A$ is obtained
as an inverse limit are all nilpotent, then
every homomorphism from $A$ to a finite-dimensional
$\!k\!$-algebra $B$ is continuous.
\end{theorem}

\begin{proof}
(i): Given $n,$ we wish to prove $A_{(n)}$ open.
Since $A_{(1)}=A,$ we may assume $n>1.$

By the fact that $A$ is the closure of $S,$
and continuity of the operations of $A,$ we have
\begin{equation}\begin{minipage}[c]{23pc}\label{d.clS_()}
The closure of $S_{(n)}$ contains $A_{(n)}.$
\end{minipage}\end{equation}

Let $S$ be generated by $g_1,\dots,g_r.$
Then $S_{(n)}$ is spanned as a $\!k\!$-vector
space by (variously bracketed)
products of these elements of lengths $\geq n.$
Let us understand an ``$\!m\!$-separating product'' of
$g_1,\dots,g_r$ to mean an element of $S$ obtained by
substituting one of $g_1,\dots,g_r$
for each indeterminate in an $\!m\!$-separating monomial.
Taking a $d$ such that $\V$ is $\![1,1{+}d\,]\!$-separative,
Lemma~\ref{L.n+d-sep} (with $n\,{-}\,1$ in place of the
$n$ of that lemma) tells us that
\begin{equation}\begin{minipage}[c]{23pc}\label{d.S=sum}
$S_{(n)}$ is spanned as a $\!k\!$-vector space by
$\![n{-}1,n{-}1{+}d\,]\!$-separating products of $g_1,\dots,g_r$
of lengths~$\geq n.$
\end{minipage}\end{equation}

Let $U_{n,d}$ denote the finite set of all monomials
$u(x_1,\dots,x_r,y)$ of lengths $n,\dots,n{+}d$ in $r{+}1$
indeterminates $x_1,\dots,x_r,y,$ in which $y$ appears exactly once.
The point of this is that if $w$ is any
$\![n{-1},n{+}d{-1}\,]\!$-separating monomial
in $x_1,\dots,x_r,$ as in~(\ref{d.S=sum}), this means
that we can choose a submonomial $w'$ such that
$\length(w)-\length(w')\in [n{-}1,n{+}d{-}1\,];$
hence $w$ can be written $u(x_1,\dots,x_r,w')$ for some $u\in U_{n,d}.$
To each $u=u(x_1,\dots,x_r,y)\in U_{n,d}$ let us
associate the $\!k\!$-linear map $f_u: A\to A$ taking
$a\in A$ to $u(g_1,\dots,g_r,a).$
Then from~(\ref{d.S=sum}) we conclude that
\begin{equation}\begin{minipage}[c]{23pc}\label{d.S=sumf}
$S_{(n)}\ =\ \sum_{u\in U_{n,d}}\,f_u(S).$
\end{minipage}\end{equation}

Now the closure of the right-hand side of~(\ref{d.S=sumf}) is
\begin{equation}\begin{minipage}[c]{23pc}\label{d.sumUnd}
$\sum_{u\in U_{n,d}}f_u(A)$
\end{minipage}\end{equation}
by continuity of the $f_u,$ and Lemmas~\ref{L.lcp<=}(iii),
\ref{L.lcp=>}(iii) and~\ref{L.lcp=>}(v);
and this sum~(\ref{d.sumUnd}) is clearly contained in $A_{(n)}.$
On the other hand, as noted in~(\ref{d.clS_()}), the closure of
the left-hand of~(\ref{d.S=sumf}) contains $A_{(n)}.$
So we have equality:
\begin{equation}\begin{minipage}[c]{23pc}\label{d.A_(n)=clS_(n)}
$A_{(n)}$ is the closure of $S_{(n)}$ in $A;$ in particular,
it is a closed subspace.
\end{minipage}\end{equation}

To show it is open, let $C_n$ be the $\!k\!$-subspace of $S$
spanned by all monomials of length $<n$ in $g_1,\dots,g_r.$
Since $g_1,\dots,g_r$ generate $S,$ we have
\begin{equation}\begin{minipage}[c]{23pc}\label{d.S=C+}
$S\ =\ C_n+S_{(n)}.$
\end{minipage}\end{equation}
Since $C_n$ is finite-dimensional, it is closed
by Lemmas~\ref{L.lcp<=}(i) and~\ref{L.lcp=>}(iii),
so taking the closure of~(\ref{d.S=C+}), we get $A=C_n+A_{(n)}$
by Lemma~\ref{L.lcp=>}(v).
Hence the closed subspace $A_{(n)}$ has finite codimension in $A,$
so by Lemma~\ref{L.lncp/fd}(ii), it is open, giving
statement~(i) of our theorem.

The remaining assertions easily follow.
Indeed, a homomorphism $f$ of $A$
into a nilpotent algebra $B$ has nilpotent image, hence
has some $A_{(n)}$ in its kernel, hence that kernel is open,
so $f$ is continuous.
If $A$ is in fact an inverse limit of nilpotent algebras,
then by \cite[Theorem~10(iii)]{pro-np}, its image under any
homomorphism to a finite-dimensional algebra is nilpotent,
so (ii) yields (iii).
\end{proof}

{\em Remark:}  In the hypothesis of the above theorem, the
condition that the algebras of which
$A$ is an inverse limit be finite-dimensional
is not needed for conclusion~(iii).
For if $A$ is an inverse limit of algebras $A_i,$
then replacing each of these by the image of $A$ therein,
we may assume that $A$ maps surjectively to each of them.
Any $S$ dense in $A$ in the inverse
limit topology will then also map surjectively to each $A_i;$
hence if $S$ is finitely generated, so are the $A_i.$
Now if the $A_i$ are nilpotent, as assumed in~(iii), then
finite generation implies finite-dimensionality.

Let us note the consequence of the above theorem for homomorphisms
among inverse limit algebras.

\begin{corollary}\label{C.lim->lim}
Suppose that $A$ is the inverse limit of a system of $\!k\!$-algebras
$(A_i)_{i\in I}$ in a separative variety $\V,$
that $B$ an inverse limit of an arbitrary system of
$\!k\!$-algebras $(B_j)_{j\in J}$
\textup{(}in each case with connecting morphisms,
which we do not show\textup{);} and that $A$
has a finitely generated dense subalgebra.

Then if either all the $A_i$ are finite-dimensional and all
the $B_j$ nilpotent, or all the $A_i$ are nilpotent and all
the $B_j$ finite-dimensional, then every algebra homomorphism
$A\to B$ is continuous in the inverse limit topologies on $A$ and $B.$
\end{corollary}

\begin{proof}
A basis of open subspaces of $B$ is given by the kernels of
its projection maps to the $B_j,$ so it will suffice to show
that the inverse image of each of those kernels under any
homomorphism $f:A\to B$ is open in $A.$
Such an inverse image is the kernel of the composite
homomorphism $A\to B\to B_j.$
If the $A_i$ are finite-dimensional and the $B_j$ nilpotent,
this composite falls under case~(ii) of Theorem~\ref{T.main},
while if the $A_i$ are nilpotent and the $B_j$ finite-dimensional,
it falls under case~(iii) (as adjusted by the above Remark).
In either case, the continuity given by the theorem
means that the kernel of the above composite map is open, as required.
\end{proof}

\section{Separativity of some varieties, familiar and unfamiliar}\label{S.assoc&&}

Clearly, any monomial $w$ of length $>1$ is congruent modulo the
consequences of the
{\em associative} identity to a $\!1\!$-separating monomial $w'\,x_i.$
So the variety of associative $\!k\!$-algebras is
$\![1,1]\!$-separative, and Theorem~\ref{T.main} applies to it.

The variety of Lie algebras is also $\![1,1]\!$-separative,
but the proof is less trivial.
When one works out the details, one sees that the
argument embraces the associative case as well:

\begin{lemma}\label{L.8}
Suppose $\V$ is the variety of associative algebras, the variety of
Lie algebras, or generally, any variety of $\!k\!$-algebras
satisfying identities modulo which each of the monomials
$(xy)z,$ $z(xy)$ is congruent to some linear combination of
the eight monomials having $x$ or $y$ as ``outside'' factor:
\begin{equation}\begin{minipage}[c]{23pc}\label{d.8}
$x(yz),\ \ x(zy),\ \ (yz)x,\ \ (zy)x,
\ \ y(xz),\ \ y(zx),\ \ (xz)y,\ \ (zx)y.$
\end{minipage}\end{equation}

Then $\V$ is $\![1,1]\!$-separative.
\end{lemma}

\begin{proof}
Let $w$ be a monomial of degree $>1$
which we wish to show congruent modulo the identities of $\V$
to a linear combination of $\![1,1]\!$-separating,
i.e., $\!1\!$-separating, monomials.
Let us write $w=w'w'',$ and induct on
$\min(\length(w'),\,\length(w'')).$

If that minimum is $1,$ then $w$ is itself $\!1\!$-separating.
In the contrary case, assume without loss of generality
(by the left-right symmetry of our hypothesis and conclusion) that
$1<\length(w')\leq\length(w''),$
and let us write the shorter of these factors,
$w',$ as $w_1 w_2,$ and rename $w''$ as $w_3,$
so that $w=(w_1 w_2)w_3.$
Putting $w_1,$ $w_2$ and $w_3$ for $x,$ $y$ and $z$
in the identity involving $(xy)z$ in the hypothesis of the lemma, we
see that $w=(w_1 w_2)w_3$ is congruent modulo
the identities of $\V$ to a $\!k\!$-linear combination of
products of $w_1,\ w_2,\ w_3$ in each of which $w_1$ or $w_2$
is the ``outside'' factor.
Since $w_1$ and $w_2$ each have length $<\length(w'),$
our inductive hypothesis is applicable to the resulting products,
so we may reduce them
to linear combinations of $\!1\!$-separating monomials,
completing the proof of the general statement of the lemma.

For $\V$ the variety of associative
algebras, the associative identity clearly yields the
stated hypothesis, while if $\V$ is the variety of Lie
algebras, two versions of the Jacobi identity, one expanding
$[[x,y],z]$ and the other $[z,[x,y]],$ together do the same.
\end{proof}

The case of Jordan algebras is more complicated; but when one
works it out, one sees a pattern of which the preceding
lemma is the $d=0$ case, while Jordan algebras come under $d=1.$
So we may consider the preceding lemma and its proof as a warm-up for
the next result, giving the general case.

We note an easy fact that we will need in the proof.
Let $w$ be a monomial of length $n.$
If $n>1,$ we can write it as a product of two submonomials;
if $n>2$ we may write one of those two as such a product, and hence
get $w$ as a bracketed product of three submonomials; and so on.
We conclude by induction that
\begin{equation}\begin{minipage}[c]{23pc}\label{d.m<n}
For every positive $m\leq\length(w),$ we can write $w$
as a bracketed product of exactly $m$ submonomials.
\end{minipage}\end{equation}

We can now prove

\begin{proposition}\label{L.general_d}
Let $\V$ be a variety of $\!k\!$-algebras, and $d$ a natural number.
Then a sufficient condition for $\V$
to be $\![1,1{+}d\,]\!$-separative is
that, for every monomial $u$ obtained by bracketing the ordered string
$x_1\dots x_{d+2}$ of $d+2$ indeterminates, each of the monomials
$uz,$ $zu$ in $d+3$ indeterminates $x_1,\,\dots,\,x_{d+2},\,z$
be congruent modulo the homogeneous identities of $\V$
to a linear combination of monomials of the form $u'u'',$
in which {\em one} of $u',$ $u''$ is a product \textup{(}in
some order, with some bracketing\textup{)}
of a {\em proper} nonempty subset of $x_1,\dots,x_{d+2}.$
\textup{(}Thus, the other factor will be a product of $z$ and those
of the $x_m$ not occurring in the abovementioned product.\textup{)}

In particular, the preceding lemma is the case $d=0$ of this result.

The variety of Jordan algebras over a field
of characteristic not $2$ satisfies this criterion with $d=1.$
\end{proposition}

\begin{proof}[Sketch of proof]
Following the pattern of the proof of the preceding lemma,
assume $w$ is a monomial of degree $>1$ that we want to express as a
linear combination of $\![1,1{+}d\,]\!$-separating monomials.
If it is not already $\![1,1{+}d\,]\!$-separating, we write $w=w'w'',$
note that $\min(\length(w'),\,\length(w''))\geq d+2,$
and induct on that minimum.
Assuming without loss of generality that $w'$ has that minimum
length, we use~(\ref{d.m<n}) to write $w'=u(w_1,\dots,w_{d+2}),$
for submonomials $w_1,\dots,w_{d+2}$ of $w'.$
We now apply to $w=w'w''=u(w_1,\dots,w_{d+2})\,w''$ an
identity of the sort described in our hypothesis for $uz,$
putting $w_1,\dots,w_{d+2}$ for $x_1,\dots,x_{d+2},$
and $w''$ for $z,$
and note that our inductive hypothesis applies to each monomial in
the resulting expression, completing the proof of our main statement.

It is easy to see that for $d=0,$ this result is equivalent
to that of the preceding lemma.

To see that the variety of Jordan algebras has the indicated
property with $d=1,$ first note that modulo
relabeling of $x_1,\,x_2,\,x_3,$ and consequences of the
commutative identity (satisfied by Jordan algebras),
all the monomials $uz$ and $zu$ for which
we must verify that property are congruent to $z(x_1(x_2\,x_3)).$
Hence we need only verify it for that monomial.

To do so, we take the Jordan identity
\begin{equation}\begin{minipage}[c]{23pc}\label{d.Jordan}
$(xy)(xx)=x(y(xx)),$
\end{minipage}\end{equation}
make the substitutions $x=x_2+x_3+z$ and $y=x_1,$ and take the part
multilinear in these four indeterminates.
Up to commutativity, this has only one term with
$z$ ``on the outside'', namely $z(x_1(x_2 x_3)),$ which occurs on
the right-hand side.
It occurs twice, but since we are
assuming $\mathrm{char}(k)\neq 2,$ we can
divide out by $2$ (which is in fact the multiplicity of every term
occurring, due to the presence of a single $(xx)$ on each side).
The resulting identity expresses $z(x_1(x_2 x_3))$ in the desired form.
\end{proof}

(In the case we have excluded from the final statement of the
proposition, where $\mathrm{char}(k)=2,$
Jordan algebras are usually defined to involve operations
quadratic in one of the variables, rather than just the bilinear
multiplication; hence they fall outside the scope of this note.
If one wants a concept of Jordan algebra over such a $k$
involving only the bilinear multiplication,
it would be natural to include among the identities
of that operation the one
gotten by taking the identity in $x_1,\,x_2,\,x_3,\,z$ obtained
by multilinearization above,
written with integer coefficients, dividing all these
coefficients by $2,$ and then reducing modulo~$2.$
If such a definition is used, our argument for Jordan
algebras is valid without restriction on the characteristic.)

One may ask whether the condition of the above proposition is
necessary as well as sufficient for  the stated conclusion.
It is not.
Using the same general approach as in the above two proofs
(but noting that in proving the final statement below,
one does not have recourse to left-right symmetry),
the reader should find it easy to verify

\begin{lemma}\label{L.()()}
The variety $\V$ of $\!k\!$-algebras defined by the identity
\begin{equation}\begin{minipage}[c]{23pc}\label{d.()()}
$(x_1 x_2)\,(x_3 x_4)\ =\ 0$
\end{minipage}\end{equation}
is $\![1,1]\!$-separative.

More generally, this is true of any variety $\V$ such that modulo
homogeneous identities of $\V,$ the monomial $(x_1 x_2)\,(x_3 x_4)$
is congruent to a linear combination of monomials of the
forms $u x_3$ and $u x_4.$\qed
\end{lemma}

\section{Counterexamples}\label{S.cegs}

Let us now construct a $\!k\!$-algebra $A$ which is an inverse limit of
finite-dimensional $\!k\!$-algebras, and has a finitely generated
subalgebra $S$ dense in the pro-discrete topology, but which does
not lie in a separative variety, and for which
the conclusions of Theorem~\ref{T.main} fail.

We need to arrange that for some $n,$ $S_{(n)}$ is {\em not}
the sum of the images of finitely many linear polynomial
operations on $S.$
Hence we need to build up, out of a finite generating
set for $S,$ an infinite family of monomials whose sums of
products will create this ``problem'' in $S_{(n)}.$
The smallest number of generators that might work is $1,$
and the smallest $n$ is $2.$
It turns out that we can attain these values.

Starting with a single generator $p,$ let
$p^2=q_1,$ and recursively define $p\,q_m=q_{m+1}.$
Our $S$ will be spanned by $p,$ these elements $q_m,$ and the
products $q_m q_n = r_{mn}.$
Letting all products other than these be $0,$
we can describe $S$ abstractly as having a $\!k\!$-basis of elements
\begin{equation}\begin{minipage}[c]{23pc}\label{d.pqr}
$p,\quad q_m,\quad r_{mn,}\quad (m,n\geq 1),$
\end{minipage}\end{equation}
and multiplication
\begin{equation}\begin{minipage}[c]{23pc}\label{d.pqr_mult}
$p\,p\,=\,q_1,\quad p\,q_m\,=\,q_{m+1},\quad q_m\,q_n\,=\,r_{mn},$\quad
with all other products of basis elements $0.$
\end{minipage}\end{equation}

For every $i>0,$ $S$ has a finite-dimensional nilpotent
homomorphic image $S_i$ defined by setting to zero all
$q_m$ with $m\geq i$ and all $r_{mn}$ with $\max(m,n)\geq i.$
Let $A$ be the inverse limit of the system
$\dots\to S_{i+1}\to S_i\to\dots\to S_1.$
This consists of all formal infinite linear combinations
of the elements~(\ref{d.pqr}), with multiplication still formally
determined by~(\ref{d.pqr_mult}).

Now if we multiply two elements $a,b\in A,$ the array of coefficients
of the various $r_{mn}$ in the product, arranged in an infinite matrix,
will clearly be given by the product of the column formed
by the coefficients of the $\!q$'s in the element $a,$ and the
row formed by the coefficients of
the $\!q$'s in the element $b.$
Hence it will have rank~$\leq 1,$ where we define the
rank of an infinite matrix as the supremum of the ranks of
its finite submatrices.
In a linear combination of $d$ such products, the corresponding matrix
of coefficients may have rank as large as $d,$ but
we see that in no element of $A_{(2)}$ will it have infinite rank.

The set of $a\in A$ such that the matrix of coefficients
in $a$ of the $r_{mn}$ has finite rank forms a proper
$\!k\!$-subspace of $A;$ e.g., it does not contain the ``diagonal''
element $\sum_m r_{mm}.$
Thus (given the Axiom of Choice) there is a nonzero
linear functional $\varphi:A\to k$ annihilating that subspace.
Let $k\,\varepsilon$ denote a $\!1\!$-dimensional
$\!k\!$-algebra with zero multiplication (i.e.,
let $\varepsilon^2=0),$ and define a $\!k\!$-linear map
$f:A\to k\,\varepsilon$ by $f(a)=\varphi(a)\varepsilon.$
From the fact that $\varphi(A_{(2)})=\{0\}$ and the
relation $\varepsilon^2=0,$ we see that $f$ is an algebra homomorphism.
Since $\ker f$ contains all
finite linear combinations of the $p,\ q_m$ and $r_{mn},$
it has all of $A$ as closure.
But it is not all of $A,$ so $f$ is not continuous.

By Theorem~\ref{T.main}, this $A$ cannot lie
in a separative variety of $\!k\!$-algebras.
However, it does lie in the variety determined
by the rather strong identities
\begin{equation}\begin{minipage}[c]{23pc}\label{d.pqr_ids}
$((x_1 x_2)\,x_3)\,x_4\,=\,0,\quad x_4((x_1 x_2)\,x_3)\,=\,0.$
\end{minipage}\end{equation}
Indeed, substituting any elements of $A$ for $x_1$
and $x_2,$ we find that $x_1 x_2$ yields a
formal $\!k\!$-linear combination of $\!q$'s and $\!r$'s only.
Multiplying this on the right by an arbitrary element gives a
formal linear combination of $r$'s only; and this annihilates
everything on both the right and the left.
(Contrast Lemma~\ref{L.()()}.)

We summarize this construction as

\begin{example}\label{E.pqr}
Let $A$ be the linearly compact $\!k\!$-algebra of all formal
infinite linear combinations of basis elements\textup{~(\ref{d.pqr})},
with multiplication determined by\textup{~(\ref{d.pqr_mult})}, and
$S$ the subalgebra of $A$ generated by $p.$
Then $S$ is dense in $A,$ and $A$ is an inverse limit of
finite-dimensional nilpotent homomorphic images $S_i$ of $S,$ and
satisfies the identities\textup{~(\ref{d.pqr_ids})};
but $A$ admits a discontinuous homomorphism $f$ to the
$\!1\!$-dimensional square-zero $\!k\!$-algebra $k\,\varepsilon.$\qed
\end{example}

We can get an example $A^\mathrm{comm}$ with similar properties,
but with commutative multiplication, if we modify the description
of the above algebra $A$ by supplementing each relation $p\,q_m=q_{m+1}$
with the relation $q_mp=q_{m+1},$ and taking $r_{mn}$ and $r_{nm}$
to be alternative symbols for the same basis element,
for all $m$ and $n.$
(If $\mathrm{char}(k)\neq 2,$ $A^\mathrm{comm}$ can be obtained from
the algebra $A$ of the above example by using the
symmetrized multiplication $x*y=xy+yx,$ and passing to
the closed subalgebra of $A$ generated by $p$ under that operation.)
In this algebra, the matrix gotten by starting with the
matrix of coefficients of the
$r_{mn}$ (now a symmetric matrix), and doubling the
entries on the main diagonal, will have rank $\leq 2$
for any product $ab,$ so on every element of $(A^\mathrm{comm})_{(2)},$
its rank will again be finite.
We deduce as before that this algebra
admits a discontinuous homomorphism to $k\,\varepsilon;$
we also note that it satisfies the identities
\begin{equation}\begin{minipage}[c]{23pc}\label{d.pqr_ids_cm}
$x_1 x_2\,=\,x_2 x_1,\quad
((x_1 x_2)(x_3 x_4))\,x_5\,=\,0.$
\end{minipage}\end{equation}

We can likewise get a version $A^\mathrm{alt}$ of our construction that
satisfies the alternating identity $x^2=0,$ again by an easy
modification of the algebra of Example~\ref{E.pqr},
or (this time without any
restriction on the characteristic), by taking an appropriate closed
subalgebra of that algebra under the operation $x*y=xy-yx.$
In this case, we can't have a relation $p^2=q_1;$ rather,
$S^\mathrm{alt}$ is generated by the two elements $p$ and $q_1.$
We find that $A^\mathrm{alt}$ satisfies the identities
\begin{equation}\begin{minipage}[c]{23pc}\label{d.pqr_ids_anti}
$x_1^2=0,\quad ((x_1 x_2)(x_3 x_4))\,x_5\,=\,0.$
\end{minipage}\end{equation}

We end this section with an easier sort of example, showing the
need for the assumption in Theorem~\ref{T.main} that $A$ have
a finitely generated dense subalgebra.
Let $k\,\varepsilon$ again denote the $\!1\!$-dimensional
zero-multiplication algebra.

\begin{example}\label{E.k^I}
For any infinite set $I,$ the zero-multiplication algebra
$(k\,\varepsilon)^I$
\textup{(}which is the inverse limit of the finite-dimensional
zero-multiplication algebras $(k\,\varepsilon)^{I_0}$
as $I_0$ runs over the finite subsets of $I,$ and
is trivially associative, Lie,
etc.\textup{)}\ admits discontinuous homomorphisms to
the $\!1\!$-dimensional zero-multiplication $\!k\!$-algebra
$k\,\varepsilon.$
\end{example}

\begin{proof}
Clearly any linear map between zero-multiplication algebras
is an algebra homomorphism; and there exist discontinuous linear maps
$(k\,\varepsilon)^I\to k\,\varepsilon.$
E.g., since there is no continuous linear extension of
the partial homomorphism taking every element of finite support
to the sum of its nonzero components, any
linear extension of that map will be discontinuous.
\end{proof}

In contrast, if we consider algebra homomorphisms $f$ such
that the image algebra $f(A)$ has nonzero multiplication, then
there are strong restrictions on examples in which, as above,
the domain of $f$ is a direct product algebra.
Namely, it is shown in \cite[Theorem~19]{prod_Lie1}
\cite[Theorem~9(iii)]{prod_Lie2} that if $k$ is an
infinite field, and $f$ a surjective homomorphism from a direct
product $A=\prod_I A_i$ of $\!k\!$-algebras to a finite-dimensional
$\!k\!$-algebra $B,$ and if $\mathrm{card}(I)$ is less than all
uncountable measurable cardinals (a condition that
is vacuous if no such cardinals exist), then
writing $Z(B)=\{b\in B\mid bB=Bb=\{0\}\,\},$ the
composite homomorphism $A\to B\to B/Z(B)$ is always continuous in the
product topology on $A$ (though the given homomorphism
$f:A\to B$ may not be).

\section{Some questions}\label{S.qs}

The result quoted above suggests

\begin{question}\label{Q.B/Z}
If $k$ is an infinite field, and $A$ an inverse limit of
of $\!k\!$-algebras $(A_i)_{i\in I}$
such that the indexing partially ordered set $I$
has cardinality less than any uncountable measurable cardinal,
can one obtain results like Theorem~\ref{T.main}\textup{(ii)}
and~\textup{(iii)} for the composite map
$A\to B\to B/Z(B)$ if $f:A\to B$ is surjective, without the
requirement that $A$ have a finitely generated dense subalgebra $S,$
and/or without the hypothesis that it lie in separative variety?
\end{question}

A different (if less interesting) way to achieve continuity, if $A$
does not have a dense finitely generated subalgebra,
might be to refine the topology in which we try to prove our maps
continuous.
The topology on $A$ defined in the next question is such that a linear
map is continuous under it if and only if its restrictions
to all ``topologically
finitely generated'' subalgebras $A'\subseteq A$ are continuous in
the restriction of the inverse limit topology.
The question asks whether this topology is a reasonable one.

\begin{question}\label{Q.weaker}
Suppose $A$ is an inverse limit of finite-dimensional algebras,
and we define a new linear topology on $A$ by taking for the
open subspaces those
subspaces $U$ whose intersections with the closures $A'$ of all
finitely generated subalgebras $S'$ of $A$ are relatively open
in $A'$ under the pro-discrete topology on $A.$

Will the multiplication of $A$ be bicontinuous in this topology?
\end{question}

For $A=(k\,\varepsilon)^I$ as in Example~\ref{E.k^I}, the topology
described above is the discrete topology on~$A.$

Recall next that in each of the results of \S\ref{S.assoc&&},
separativity was obtained from some finite family of identities.
We may ask whether this is a general phenomenon.

\begin{question}\label{Q.ids_fr_fin_many}
Suppose $\V$ is a separative variety of $\!k\!$-algebras.
Will some overvariety $\V'$ determined by finitely many
identities still be separative?

If this is so, will there be such a $\V'$
which is $\![1,1{+}d\,]\!$-separative
for the least $d$ for which $\V$ has that property?
\end{question}

In Example~\ref{E.pqr}, the fact that $A_{(2)}$ was not closed
in $A$ was related to the fact that
it consisted of sums of arbitrarily large numbers of elements of
\begin{equation}\begin{minipage}[c]{23pc}\label{d.ab}
$\{a\,b\mid a,\,b\in A\}.$
\end{minipage}\end{equation}
One may ask whether the set~(\ref{d.ab}) (not itself a vector
subspace) is nevertheless
closed in our topology on $A$ (though a positive
answer would not lead to any obvious improvement of our results).
More generally,

\begin{question}\label{Q.ab}
If $A,$ $B$ and $C$ are linearly compact vector
spaces, and $f:A\times B\to C$ is a bicontinuous
bilinear map, must $\{f(ab)\mid a\in A,\ b\in B\}$ be closed in $C$?
\end{question}

(Examination of the algebra of Example~\ref{E.pqr}
shows that for that map, the answer is yes.
What this says is that one can test whether an element of $A$ has the
form $a\,b$ by looking at its coordinates finitely many at a time.)

We saw in Lemma~\ref{L.lcp=>}(iv)
that every linearly compact vector space
is an inverse limit of finite-dimensional discrete vector spaces.
Is every linearly compact algebra (i.e., every linearly compact vector
space made an algebra using a bicontinuous multiplication) an inverse
limit of finite-dimensional algebras?
For associative algebras -- yes; in general -- no!
Indeed, it is not true for Lie algebras \cite[Example~25.49]{coalg}.
So we ask

\begin{question}\label{Q.just_lin_cp}
Does either statement of Corollary~\ref{C.lim->lim} remain
true if the assumption that $A,$ respectively $B,$ is an
inverse limit of finite-dimensional algebras is weakened
to say merely that it is a linearly compact algebra?
\end{question}

Recall also that Corollary~\ref{C.lim->lim} has the peculiar
hypothesis that either the $A_i$ are finite-dimensional and
the $B_j$ nilpotent, or the $A_i$ are nilpotent and
the $B_j$ finite-dimensional.
Of the two other possible ways of distributing ``finite-dimensional''
and ``nilpotent'' among the $A_i$ and the $B_j,$ the arrangement
that puts both conditions on the $A_j$ and
no such condition on the $B_i$ certainly does not imply continuity;
for one can take a nondiscrete $A$ arising in this way, and let
$B$ be the same algebra with the discrete topology, regarded as an
inverse limit in a trivial way.
But I do not know about the reverse arrangement.

\begin{question}\label{Q.Bj_fd_np}
If in the last sentence of Corollary~\ref{C.lim->lim} we
instead assume that the $B_j$ are finite-dimensional and
nilpotent \textup{(}with no such condition on the $A_i),$ can we still
conclude that every algebra homomorphism $A\to B$ is continuous?
\end{question}

Everything we have done so far has depended on pro-nilpotence;
but we may ask

\begin{question}\label{Q.non_np}
Is the analog of Theorem~\ref{T.main}\textup{(iii)} true
with nilpotence either replaced by other conditions
\textup{(}e.g., solvability, some version of semisimplicity,
etc.\textup{)}, or dropped altogether?
\end{question}

The generalization of Serre's result on pro-$\!p\!$ groups analogous
to the result asked for above, i.e., with ``pro-$\!p\!$'' generalized
to ``profinite,'' has, in fact,
been proved \cite{NN+DS.I}, \cite{NN+DS.II}.
The proof of this deep result
uses the Classification Theorem for finite simple groups.

In connection with Question~\ref{Q.non_np},
let us recall briefly the meaning of solvability for
a general $\!k\!$-algebra $A;$
it is the straightforward extension of the condition
of that name arising in the theory of Lie algebras
\cite[p.17]{Schafer}:
One defines subspaces $A^{(n)}$ $(n=0,1,\dots)$ of $A$ recursively by
\begin{equation}\begin{minipage}[c]{23pc}\label{d.der_series}
$A^{(0)}=A,\qquad A^{(n+1)}= A^{(n)}\,A^{(n)},$
\end{minipage}\end{equation}
and calls $A$ solvable if $A^{(n)}=\{0\}$ for some $n.$

A difference between nilpotence and solvability which may be
relevant to the above question is that a
finitely generated solvable algebra, unlike a
finitely generated nilpotent algebra, can be infinite-dimensional.
(E.g., the $S$ of Example~\ref{E.pqr} is solvable.
There exist similar examples among Lie algebras.)
So it might be necessary to make finite codimensionality of
the $S^{(n)}$ in $S$ an additional hypothesis in a version of
that theorem for solvable algebras, if such a result is true.

\section{Questions on subalgebras of finite codimension}\label{S.salg_q}

The referee has raised the following interesting question:

\begin{question}\label{Q.salg}
In cases where we have proved or conjectured
that all ideals of finite codimension in an algebra $A$ must be open
\textup{(}Theorem~\ref{T.main}\textup{(iii)} and
Question~\ref{Q.non_np}\textup{)},
can one show more generally that all {\em subalgebras}
of finite codimension in $A$ are open?
\end{question}

While the condition that an ideal $I$ be open means that the
homomorphism $A\to A/I$ is continuous with respect to the discrete
topology on $A/I,$ the condition that a subalgebra
be open has no similar interpretation.
However, since a basis of open
subspaces of $A=\lm A_i$ is given by the kernels of the projection
maps $A\to A_i,$ a subalgebra will be open if and only if it
contains one of these open ideals; equivalently, if and only if
it is the inverse image in $A$ of a subalgebra of one of the~$A_i.$

Mekei \cite{AM} shows that in an {\em associative} algebra $A,$
any subalgebra of finite codimension $n$ contains an
ideal of finite codimension $\leq n(n^2+2n+2)$ in $A.$
Hence for associative $A,$ we can indeed get results
of the sort asked for: Theorem~\ref{T.main}(iii) and Mekei's result
together imply that in a topologically finitely generated
inverse limit of finite-dimensional nilpotent associative
algebras, every subalgebra of finite codimension is open.
Positive results on Question~\ref{Q.non_np} would likewise
yield further results of this sort in the associative case.

Riley and Tasi\'{c} \cite[Lemma 2.1]{DR+VT} prove for $\!p\!$-Lie
(a.k.a.\ restricted Lie) algebras (which lie outside the scope of
this note) a result analogous to Mekei's (though with an
exponential bound in place of $n(n^2+2n+2)).$
But as they note in \cite[Example 2.2]{DR+VT}, the
corresponding statement fails for ordinary Lie algebras over a field
$k$ of characteristic~$\!0:$
The Lie algebra of derivations of
$k[x]$ spanned by the operators $x^n\,d/dx$ $(n\geq 0)$ has
a subalgebra $B$ of codimension $1,$ spanned by those operators
with $n>0;$ but $A$ is simple, so $B$ cannot contain an
ideal of $A$ of finite codimension.

This example can be completed to a linearly compact one:
In the Lie algebra $A$ of derivations on the formal power series
algebra $k[[x]]$ given by the operators $p(x)\,d/dx$ $(p(x)\in k[[x]]),$
the operators such that $p(x)$ has constant term
$0$ again form a subalgebra $B$ of codimension~$1;$
but again, $A$ is simple.
An example not limited to characteristic~$0$
can be obtained from \cite[Example~25.49]{coalg}.
There, we don't get simplicity, but still have too few ideals
for there to be one of finite codimension in a certain $B.$

These examples show that we do not have a result like Mekei's
in the variety of Lie algebras over $k,$ and hence cannot obtain
positive answers to cases of Question~\ref{Q.salg} for that
variety in the way we did for associative algebras.
But this does not say that such results don't hold.
Indeed, the above examples used Lie algebras
with a paucity of open ideals of
finite codimension, while a pro-finite-dimensional algebra necessarily
has a neighborhood basis of the identity consisting of such ideals.
So the answer
to Question~\ref{Q.salg}, even for Lie algebras, remains elusive.

It would also be of interest to know for what varieties
an analog of Mekei's result does hold:

\begin{question}\label{Q.Mekei}
For what varieties $\V$ of $\!k\!$-algebras is it true that
any subalgebra $B$ of finite codimension in an
algebra $A\in\V$ contains an ideal of $A$ of finite codimension?
\end{question}

% - - - - - - - - - - - - - - - - - - - - - - - - - - - - - -

\end{document}